\documentclass[
final
]{dmtcs-episciences}

\usepackage[utf8]{inputenc}
\usepackage{subfigure}

\usepackage{amsthm}
\usepackage{amsfonts}
\usepackage{amsmath}
\usepackage{amssymb}
\usepackage{theoremref}
\newtheorem{theorem}{Theorem}
\newtheorem{corollary}{Corollary}

\newtheorem{lemma}{Lemma}
\usepackage[round]{natbib}

\author{Behrouz Zolfaghari\affiliationmark{1}
  \and Mehran S. Fallah\affiliationmark{1}\thanks{Corresponding author}
  \and Mehdi Sedighi\affiliationmark{1}}
\title[S-restricted compositions revisited]{S-restricted compositions revisited}
\affiliation{
  Department of Computer Engineering and Information Technology, Amirkabir University of Technology, Tehran, Iran}
\received{2016-6-29}
\revised{2017-3-2}
\accepted{2017-03-17}

\begin{document}
\publicationdetails{19}{2017}{1}{9}{1522}
%\publicationdetails{VOL}{2015}{ISS}{NUM}{SUBM}
\maketitle
\begin{abstract}
An $S$-restricted composition of a positive integer $n$ is an ordered
partition of $n$ where each summand is drawn from a given subset $S$ of
positive integers. There are various problems regarding such compositions
which have received attention in recent years. This paper is an attempt at
finding a closed-form formula for the number of $S$-restricted compositions
of $n$. To do so, we reduce the problem to finding solutions to corresponding
so-called interpreters which are linear homogeneous recurrence relations with
constant coefficients. Then, we reduce interpreters to Diophantine equations.
Such equations are not in general solvable. Thus, we restrict our attention
to those $S$-restricted composition problems whose interpreters have a small
number of coefficients, thereby leading to solvable Diophantine equations.
The formalism developed is then used to study the integer sequences related
to some well-known cases of the $S$-restricted composition problem.
\end{abstract}

%DMTCS is an open access scientific is implemented by the
%\emph{episcience} platform, see \cite{berthaud:hal-01002815} for an
%overview of the strategy. It combines high scientific and editorial
%quality with an open access policy. It is priceless, neither authors
%nor readers pay money for the access. Access is granted by giving
%episcience an irrevocable license to publish the articles, the
%copyright remains with the authors. The platform itself is run by
%French government services that do their best to warrant continuous
%access and a high quality of service.
%
%This document describes the use of the \texttt{dmtcs-episcience.cls}
%document class. It should be used
%\begin{center}
%  \emph{\textbf{for all DMTCS publications}}.
%\end{center}
%
%If you are still preparing a document for our previous \texttt{OJS}
%platform, please add \texttt{ojs} to the classes options, see
%Section~\ref{sec:options}.

\section{Introduction}
\label{sec:intro} A composition, or an ordered partition, of a positive integer
$n$ is a tuple of positive integers whose sum equals $n$. Integer compositions
have been of high interest as they have formed part of the solution in many
problems \cite{bib51, bib48, bib45, bib18}. Finding the number of compositions
of $n$ where summands belong to a subset $S$ of positive integers, known as
$S$-restricted compositions of $n$, is a challenging problem. This problem is
referred to as the $S$-restricted composition problem in this paper.

There are closed-form solutions to the $S$-restricted composition problem for
some particular subsets $S$ of positive integers. Moreover, for some classes of
$S$, it has been shown that the number of $S$-restricted compositions can be
obtained by solving specific recursive equations. Deriving a closed-form
solution for the general $S$-restricted composition problem through recursive
equations and fractional generating functions will lead to solving polynomial
equations. Moreover, the use of exponential generating functions for this
problem leads to Diophantine equations. Neither polynomial nor Diophantine
equations are in general solvable. That is why there is no closed-form formula
for the general problem.

As stated above, solving a linear homogeneous recurrence relation with constant
coefficients (LHRC) through classical methods requires solving a polynomial
equation. The degree of such a polynomial equation is equal to the largest
offset of the recurrence relation. Instead of solving such polynomial
equations, which is impossible in most cases, we derive Diophantine equations
whose number of variables is the number of offsets of the recurrence relation.
Such a transformation is quite useful for the $S$-restricted composition
problems which lead to LHRCs with a small number of offsets. To arrive at
appropriate LHRCs with acceptable number of offsets, we give two equivalent
LHRCs, which we call the interpreters of the $S$-restricted composition
problem. The proposed interpreters lead to different Diophantine equations from
which we choose the one with smaller number of variables. To illustrate the
usefulness of our technique, we then apply it to some classes of the problem.
It is worth noting that the interpreters are not limited to those given in this
paper and one may find other interpreters that are more appropriate for other
classes of the problem not studied in this paper.

To explain our technique, we first introduce some notation. We denote the set
of all $S$-restricted compositions of $n$ by $\mathcal{R}(S,n)$ and the number
of such compositions, i.e., $|\mathcal{R}(S,n)|$, by $R(S,n)$. Moreover,
$\mathcal{C}(n)$ is a shorthand for $\mathcal{R}(\mathbb{Z}^+,n)$. We denote
the empty tuple by $\epsilon$ and define $\mathcal{R}(S,0)=\{\epsilon \}$ for
$S \subset \mathbb{Z}^+$, although it is conventionally undefined. For $c \in
\mathcal{R}(S,n)$, by $\text{first}(c)$, we mean the first part of $c$. The
tuple obtained from $c$ by removing its first part is denoted by $f^-(c)$,
which will obviously be equal to $\epsilon$ if $c$ includes exactly one part.
Moreover, the tuple obtained from prepending a positive integer $a$ to $c$ is
denoted by $a;c$. The tuple $a;c$ is equal to $a$ if $c=\epsilon$.  Moreover,
the sum and the product of all elements of a tuple $T$ are represented by
$\sigma(T)$ and $\pi(T)$, respectively. For $(i_1,\ldots,i_k)\in
\mathcal{R}(S,n)$ and a vector $\mathbf{c}$, the tuple
$t=(c_{i_1},\ldots,c_{i_k})$ is called an $(S,n)$-tuple of $\mathbf{c}$ in
which $c_{i_j}$ is the $i_j$th element of $\mathbf{c}$. Moreover, when
$\mathbf{c}$ is known from the context, $(i_1,\ldots,i_k)$ is denoted by
$\lambda(t)$ and is called the index tuple of $t$. Furthermore,
$\mathcal{R}_{\mathbf{c}}(S,n)$ is the family of all $(S,n)$-tuples of
$\mathbf{c}$.

Let $A$ be a finite set of $q$ positive integers. The elements of $A$ can
naturally be arranged as $a_1 < a_2 < \ldots < a_q$. Then, the row vector
$[a_1,\: a_2,\:...,\:a_q]$ is denoted by $\mathbf{v}^A$. For a positive integer
$n$, we also define $A_{\leq n}$ as $A\cap[1:n]$. The $j$th element of a vector
$\mathbf{x}$ is noted $x_j$. The length of a vector
 $\mathbf{x}$, noted $|\mathbf{x}|$, is defined to be the number of elements of
 $\mathbf{x}$. The inner product
of vectors $\mathbf{x}$ and $\mathbf{y}$ is denoted by $\mathbf{x} \cdot
\mathbf{y}$. Finally, for two vectors $\mathbf{x}$ and $\mathbf{y}$ of the same
length, by $\mathbf{x} \geq \mathbf{y}$, we mean that $x_j \geq y_j$ for every
$j$.

The $S$-restricted composition problem can be solved through solving a
Diophantine equation. Consider the Diophantine equation
\begin{equation}\label{Dio}
\mathbf{v}^S \cdot \mathbf{x}=n
\end{equation}
in which $|\mathbf{x}|=|\mathbf{v}^S|=|S|$, that is, the number of variables
equals the cardinality of $S$. Every solution $\mathbf{x}$ to \eqref{Dio}
satisfying $\mathbf{x} \geq \mathbf{0}$ represents those $S$-restricted
compositions of $n$ in which $v^S_i$ occurs $x_i$ times where $i=1,\ldots,
|S|$. The number of such compositions is
\begin{equation*}
C(\mathbf{x})=\frac{\left ( \sum_{i=1}^{|\mathbf{x}|}x_i \right ) !}{\prod_{i=1}^{|\mathbf{x}|}x_i!}={\sum_{i=1}^{|\mathbf{x}|}x_i\choose x_1, x_2,...,x_{|\mathbf{x}|}}.
\end{equation*}
Therefore, the number of all $S$-restricted compositions of $n$ is
$\sum_{\mathbf{x} \in D} C(\mathbf{x})$ where $D$ is the set of all solutions
to \eqref{Dio} satisfying $\mathbf{x} \geq \mathbf{0}$. The number of variables
in the above equation is $|S|$. Finding a closed-form solution for a
Diophantine equation with a large number of variables is a challenging problem.
This makes it difficult to derive a closed-form formula for the number of
$S$-restricted compositions by directly reducing the problem to a Diophantine
equation similar to \eqref{Dio} if $|S|$ is large. It is also difficult to
extract a closed-form solution for the $S$-restricted composition problem via
reducing it to an LHRC and solving the LHRC via its characteristic equation. To
clarify this point, assume that solving the $S$-restricted composition problem
can be accomplished through solving the LHRC
\begin{equation} \label{interp}
f(n)=\sum\limits_{i=1}^{l}{k_i f(n-a_i)}
\end{equation}
in which $k_i \neq 0$ and $a_i \in \{1,2,..., M\}$ for some $M \in
\mathbb{Z}^+$. In such a case, Equation \eqref{interp} is referred to as an
interpreter for the $S$-restricted composition problem. Solving \eqref{interp}
leads to solving a polynomial equation of degree $\max{\{a_i\}}$. Such a
polynomial equation is not in general solvable for $\max{\{a_i\}} \geq 5$
\cite{bib19, bib22, bib38,bib23,bib24} making it impossible to solve most
well-known cases of the $S$-restricted composition problem through this
approach. This motivates us to give a method for solving \eqref{interp} via
finding solutions for Diophantine equations with $|\{a_i\}|$ variables. This
simplifies solving LHRCs in which $|\{a_i\}|$ is small. Solving any case of the
$S$-restricted composition problem for which an interpreter similar to
\eqref{interp} can be derived is consequently simplified provided that
$|\{a_i\}|$ is small .

We derive two interpreters for the $S$-restricted composition problem which
lead to closed-form solutions for the cases studied in this paper. That is, the
cases in which the set $S$ is of the form $\{a_1, a_2\}$, $\{a \ | \ a=r \ (mod
\ m)\}$, $[1:m]$, $[1:n]\setminus \{m\}$, or $[1:n]\setminus [1:m]$ where
$[i:j]$ is the subset $\{i,i+1,i+2,\ldots, j\}$ of $\mathbb{Z}^+$. One may,
however, envision other interpreters that are more appropriate in other cases.
Finding appropriate interpreters can indeed be thought of as a step towards
solving the problem.

The rest of this paper is organized as follows: Section~\ref{sec:related}
discusses previous research on related problems. Section~\ref{sec:interpreters}
derives the interpreters. Section~\ref{sec:solvLHRC} attempts to derive a
closed-form solution for the LHRC of Section~\ref{sec:interpreters} and shows
that it can be derived from the solution of a Diophantine equation whose number
of variables is less than $|S|$. Section~\ref{sec:solvLHRC} also gives
closed-form solutions for LHRCs (including the interpreters derived in
Section~\ref{sec:interpreters}) using Diophantine equations.
Section~\ref{sec:relprob} employs the results of Section~\ref{sec:solvLHRC} to
solve some instances of the $S$-restricted composition problem and finds the
elements of some well-known integer sequences. Section~\ref{sec:conclusion}
concludes the paper.

\section{Related Work}
\label{sec:related} Many applied problems involve variants of the integer
composition problem where different kinds of constraints are placed on the
structure or parts of the composition. We briefly review earlier research on
integer compositions with such constraints. As the first example, we can refer
to those considering palindromic compositions which read the same from the left
and the right \cite{bib36}. Another example is the problem of locally
restricted compositions in which every $k$ successive parts meet particular
constraints \cite{bib27,bib28,bib29}. Compositions whose parts are from a given
set and successive parts of a given length avoid particular patterns are also
investigated in the literature \cite{bib34,bib42}. The research on the same has
also led to the generating functions that yield a closed-form solution to the
problem in special cases \cite{bib31}. Generating functions have also been
derived for the number of compositions where the swap of two parts of the
composition is not relevant \cite{bib32}. Compositions avoiding some partially
ordered patterns have also been studied \cite{bib33}. Partially ordered
patterns define a partial order relation on parts of a composition.

A great deal of research has been dedicated to $S$-restricted compositions of a
given positive integer $n$ with constraints on the value and the number of
summands and the way they can be arranged in the composition. Research on the
number of such compositions has even been augmented by investigating their
probabilistic, statistical, or asymptotic properties \cite{bib17,bib44}. In
most cases, however, a closed-form formula is far from reach, and thus, the
research has not gone beyond generating functions or recursive equations that
yield the the number of compositions if solvable. It is for this reason that
some researchers have estimated \cite{bib6} and even conjectured the number of
compositions \cite{bib30}. The following is a short overview of the earlier
work on the problem of $S$-restricted composition.

In one line of research, it has been tried to derive generating functions for
the number of $S$-restricted compositions. Among the attempts made, we can
mention \cite{bib1} in which a generating function is derived for the number of
palindromic $S$-restricted compositions. We may also refer to the generating
functions presented in \cite{bib30} for the number of compositions of $n$
consisting of a specific number of parts. The functions are further used to
calculate the probability that two independently-selected random compositions
of $n$ have the same number of parts.

Another line of research in this area has focused on solving the general
problem of $S$-restricted composition, where the only constraint is that parts
belong to $S$, by deriving recursive equations. An example is the recursive
equation derived in \cite{bib5} for the general problem. The inherent
difficulty with such equations is that they are in general unsolvable. The same
paper solves the recursive equation in the special case where $S$ is a bounded
set of consecutive positive integers.

The impossibility of finding a general solution to the problem of
$S$-restricted compositions has motivated researchers to give solutions in
useful and well-known special cases. A case is where $S$ consists of two
elements, namely $1$ and $2$ \cite{bib13}, $1$ and $K$ \cite{bib14}, and $2$
and $3$ \cite{bib15}, to mention a few. In another case, $S$ is the set of
powers of a specific number \cite{bib11,bib12}. The number of compositions in
case the summands are from a given sequence of numbers has also been
investigated \cite{bib10}. The same has also been studied when the summands are
congruent to $k$ in modulo $m$ for some particular integers $k$ and $m$. For
example, it has been shown that for $k=1$ and $m=2$, the number of compositions
of $n$ is the $n$th Fibonacci number \cite{bib35}. For $k \in \{1,2\}$ and
$m=4$, the number of compositions of a given integer is an element of the
Padovan sequence \cite{bib43}. The problem has also been studied in the case
$S=\mathbb{Z}^+ \setminus \{m\}$ \cite{bib16,bib8,bib9}.

Another approach to the problem of $S$-restricted compositions is to reduce it
to some other known problems. For example, it has been shown that some cases of
the problem are in bijective correspondence with some classes of restricted
binary strings and pattern-avoiding permutations \cite{bib56}. It has also been
shown that the number of $S$-restricted compositions equals specific Fibonacci
numbers for some specific sets $S$ \cite{bib7}.

The $S$-restricted composition problem can also be studied as a special case of
the problems weighted compositions \cite{bib101,bib102,bib103,bib106},
restricted words \cite{bib105}, and colored compositions \cite{bib107,bib108}.
The solutions presented for these problems, however, defer the problem to
finding a closed-form formula for so-called extended binomial coefficients.

Having said all this, there is no closed-form formula for the number of
$S$-restricted compositions. We do not solve the general problem either, but we
take this area of research one step further. To do so, we reduce the general
problem to linear homogeneous recursive equations and Diophantine equations so
that we can find exact solutions in new families of cases.

\section{Deriving Interpreters}
\label{sec:interpreters} In this section, we derive interpreters for the
$S$-restricted composition problem. We begin with the fact that for any $c \in
\mathcal{R}(S,n)$, there exists $s\in S$ such that $\text{first}(c)=s$ and
$f^-(c)\in \mathcal{R}(S,n-s)$. It is also noted that for any $s\in S$ and
$c\in \mathcal{R}(S,n-s)$, we have $s;c\in \mathcal{R}(S,n)$. Thus, the number
of $S$-restricted compositions of $n$ satisfies
\begin{equation}\label{firstLHR}
R(S,n)=\sum\limits_{s\in S_{\leq n}}R(S,n-s)
\end{equation}
%for every $n > \max S$.
whenever $S_{\leq n}$ is not empty.

We refer to Equation \eqref{firstLHR} as the first interpreter of the
$S$-restricted composition problem. It is obvious that for $n \geq \max (S)$,
we should have the values of $R(S,i)$ for $i \in [1:\max(S)]$ as initial
conditions to solve the above equation. Let us examine how to obtain the
initial conditions. Let $\mathbf{v}^S=[s_1,\:s_2,\:\dots,\:s_{|S|}]$. $R(S,n)$
is equal to $0$ for $n<s_{1}$ and $R(S,s_{1})$ equals $1$. One can calculate
the $m$th initial value, $m=s_1+1, s_1+2, ..., \max(S)$, in their respective
order using \eqref{firstLHR} where $n$ is replaced with $m$.

The number of terms in the right side of \eqref{firstLHR} is clearly equal to
$|S_{\leq n}|$. Moreover, in Section~\ref{sec:solvLHRC}, we show that recursive
equations of the form $f(n)=\sum_{i \in I} a_i f(n-i)$ with $a_i \in
\mathbb{R}$ are solvable if we can solve linear Diophantine equations with
$|I|$ variables. Since Diophantine equations with a small number of variables
are solvable, we attempt to derive another interpreter (referred to as the
second interpreter) in this section. We show in Section~\ref{sec:relprob} that
some instances of the $S$-restricted composition problem are easier to solve
using the first interpreter and others through the second interpreter.

In order to derive the second interpreter, we first propose a procedure to
produce $\mathcal{R}(S,i)$ using $\mathcal{R}(S,i-s)$ where $s \in \{s|s \in
S_{\leq i}, s-1 \notin S\}$. The following lemma introduces the procedure.

\begin{lemma}\label{FirstFirstLemma}
For every positive integer $i$, the following procedure derives
$\mathcal{R}(S,i)$ from the sets $\mathcal{R}(S,i-s)$ where $s\in \{s|s \in
S_{\leq i},s-1 \notin S\}$.

\begin{enumerate}
\item Set $R=\emptyset$.
\item If $i \neq 1$, for each composition $c$ of $\mathcal{R}(S,i-1)$, add
    $(\text{first}(c)+1);f^-(c)$ to $R$ provided $\text{first}(c)+1 \in S$.
\item For every $s$ in $S_{\leq i}$, if $s-1 \notin S$, for each composition
    $c$ in $\mathcal{R}(S,i-s)$, add $s;c$ to $R$.
\item Set $\mathcal{R}(S,i)=R$.
\end{enumerate}
\end{lemma}
\begin{proof}
It is obvious that $R\subseteq \mathcal{R}(S,i)$. Thus, we only need to show
that $\mathcal{R}(S,i)\subseteq R$. There are two possible cases for the first
part $s$ of a given composition $c$ of $i$. If $s-1 \in S$, then
$c=(\text{first}(c')+1);f^-(c')$ for some $c' \in \mathcal{R}(S,i-1)$. If $s-1
\notin S$, $c=s;c'$ for some $c' \in \mathcal{R}(S,i-s)$.
\end{proof}

The following theorem presents the second interpreter.

\begin{theorem}\label{FirstFirstTheorem}
For every positive integer $n>1$, $R(S,n)$ satisfies the following LHRC.
\begin{equation}\label{secondLHR}
R(S,n)=R(S,n-1)-\sum_{\begin{tiny} \begin{array}{c} s-1\in S_{\leq n},\\ s\notin S \end{array} \end{tiny}} R(S,n-s) + \sum_{\begin{tiny} \begin{array}{c}s\in S_{\leq n},\\ s-1 \notin S \end{array} \end{tiny}}R(S,n-s)
\end{equation}
\end{theorem}
\begin{proof}
In the right side of \eqref{secondLHR}, the first two terms clearly count the
compositions created by the second step and the second term counts those
created by the third step of the procedure given in Lemma
\ref{FirstFirstLemma}.
\end{proof}
The number of terms in the right side of the first interpreter is equal to
$T_1=|S_{\leq n}|$ while that of the second interpreter is equal to $T_2=|S'|$
where
\begin{equation*}
S'=\{1\}\cup \{s|s-1\in S_{\leq n}, s\notin S\}\cup \{s|s\in S_{\leq n}, s-1\notin S\}.
\end{equation*}

Solving the $S$-restricted composition problem through the first interpreter is
easier if $T_1 < T_2$ and the second interpreter will lead to a simpler
resolvent if $T_1 > T_2$.

\section{Solving LHRCs through Diophantine Equations}
\label{sec:solvLHRC} In this section, we present a technique for solving LHRCs
based on Diophantine equations. We use the technique to derive closed-form
solutions for LHRCs with two or three terms because they are useful in
$S$-restricted composition problems studied in Section~\ref{sec:relprob}. The
same approach can be used to solve LHRCs with more terms.

Every LHRC can be written in the form of \eqref{interp} in which $k_i$s and
$a_i$s are called coefficients and offsets, respectively. The vectors
$\mathbf{a}=[a_1,\:a_2,\:\ldots,\:a_l]$ and $\mathbf{k}=
[k_1,\:k_2,\:\ldots,\:k_l]=[\kappa(a_1),\:\kappa(a_2),\:\ldots,\:\kappa(a_l)]$
are referred to as the offset and the coefficient vectors in which
$\kappa(a_i)$ is the coefficient of the term with offset $a_i$. The function
$\kappa$ that maps the offsets to the coefficients of \eqref{interp} is also
referred to as the representative of \eqref{interp}. Moreover,
$A=\{a_i|i\in[1:l] \}$ and $K=\{k_i|i \in [1:l] \}$ are called the offset and
coefficient sets of \eqref{interp}, respectively. It is obvious that
$\mathbf{a}=\mathbf{v}^A$. Using the representative function, \eqref{interp}
can be rewritten as follows.
\begin{equation}\label{interp2}
f(n)=\sum\limits_{i=1}^l\kappa(a_i)f(n-a_i)
\end{equation}

By one-step expansion of \eqref{interp}, we mean expanding every individual
term in the right side by \eqref{interp} itself. The LHRC obtained from
repeating one-step expansions is called an expanded version of \eqref{interp}.
Moreover, by solving \eqref{interp}, we mean finding an expanded version of it
where the right side only consists of $f(i)$s for $i\in [0:a_l-1]$. Such an
expanded version is of the form
 \begin{equation}\label{Solution}
 f(n)=\sum\limits_{i=0}^{a_l-1}V(n,i)f(i)
 \end{equation}
and is a solution to \eqref{interp} because $f(i)$s are already known for $i\in
[0:a_l-1]$ as initial conditions. The only remaining problem is to find the
coefficients $V(n,i)$. The following lemma is the first step to this end.

\begin{lemma}\label{FirstLemma}
The following procedure produces $R=\bigcup_{i\in
\left[n-\left(\max{S}-1\right):n\right]}\mathcal{R}\left(S,i\right)$.
\begin{enumerate}
\item Set $m_1=min(S)$, $m_2=\max(S)$, $Q=\{\left(s\right)|s\in S\}$, and
    $R=\emptyset$.
\item Repeat the following step $\lfloor \textstyle\frac{n}{m_1}\rfloor$
    times.

\item For every element $q$ of $Q$, if $\sigma (q)\geq n- \left( m_2-1 \right
    )$, add $q$ to $R$. Otherwise, remove $q$ from $Q$, set $h=q$, and for
    every element $s$ of $S$, add $s;h$ to $Q$.
\end{enumerate}
\end{lemma}

\begin{proof}
The above procedure\ clearly does not allow $\sigma(q)$ to get smaller than
$n-\left(m_2-1\right)$ or greater than $n$. Thus, $R$ will be a subset of
$\bigcup_{i\in
\left[n-\left(m_2-1\right):n\right]}\mathcal{R}\left(s,i\right)$. On the other
hand, $\bigcup_{i\in
\left[n-\left(m_2-1\right):n\right]}\mathcal{R}\left(s,i\right)$ is a subset of
$R$ because $R$ includes every tuple consisting of a permutation of elements of
$S$ with a total sum of $i\in \left[n-\left(m_2-1\right):n\right]$.
\end{proof}

The following lemma is another step towards calculating $V(n,i)$s.
\begin{lemma}\label{SecondLemma}
The coefficients $V(n,i)$ in the solution \eqref{Solution} for \eqref{interp}
can be obtained from
\begin{equation}\label{Vni}
 V(n,i)=\sum\limits_{\begin{tiny}\begin{array}{c}z\in  \mathcal{R}_{\mathbf{k}}(A,n-i), \\ \mbox{first}(\lambda(z))\geq a_l-i\end{array}\end{tiny}} \pi(z)
\end{equation}
\end{lemma}

\begin{proof}
Let us examine the process of expanding \eqref{interp} till reaching a solution
in the form of \eqref{Solution}. Now, consider the offset $a_i$ and its
corresponding coefficient $\kappa(a_i)=k_i$ in \eqref{interp} as tuples
$(\kappa(a_i))$ and $(a_i)$. Expanding every term in the right side of
\eqref{interp} represented by $(o_i,c_i)=(a_i,\kappa(a_i))$ replaces the term
by $\{(\kappa(a_j)o_i,c_i+a_j)|j\in[1:l]\}$. It is immediate that if
$n-(a_j+c_i) \in [0:a_l-1]$, the pair $(\kappa(a_j)o_i,c_i+a_j)$ represents
$\kappa(a_j)o_if(n-(c_i+a_j))$ which should not be expanded any more because
$f(n-(c_i+a_j))$ is an initial value. Thus, during successive steps of
expanding individual terms of the equation, every offset will be formed as
$\sigma(z_{A})$ where $z_A$ is a tuple consisting of elements of $A$. The
corresponding coefficient will also be $\pi(z_{K})$ in which $z_{K}$ is a tuple
consisting of elements of $K$. It can easily be shown through induction that
$\lambda(z_{K})$ will be equal to $z_A$ for every term in every expansion step.
 Lemma \ref{FirstLemma} states that the index tuple of every $z_{K}$
(like its corresponding $z_A$) will be an $A$-restricted composition of an
integer $n-\left(a_l-1\right)\leq i \leq n$ after
$\tau=\lfloor\frac{n}{a_1}\rfloor$ expansion steps provided that initial terms
(terms with offsets $n-\left(a_l-1\right)\leq O \leq n$) are not further
expanded. Thus, after $\tau$ steps of expansions, we reach an equation in which
only initial terms $f(i): i\in [0:l-1]$ appear in the right side. We refer to
the latter equation as the fully-expanded equation. Lemma \ref{FirstLemma} also
states that for every composition $c\in \bigcup_{i\in
\left[n-\left(m_2-1\right):n\right]}\mathcal{R}\left(s,i\right)$, there will be
a term with coefficient $\pi(z_{K})$ and offset $\sigma(z_A)$ in the
fully-expanded equation where $\lambda(z_{K})=z_A=c$ if $\text{first}(c)\geq
a_l-1$. The condition $\text{first}(c)\geq a_l-i$ guarantees that terms with
offsets $n-\left(a_l-1\right)\leq O \leq n$ are not further expanded.
\end{proof}

\begin{theorem}\label{SecondTheorem}
The following is the solution to \eqref{interp}
\begin{align}\label{DiophantineSolution}
f(n)=\sum_{\alpha\in A\setminus\{a_1\}}\sum_{i=a_l-\alpha}^{a_l-a_1-1}\kappa(\alpha)f_R(&n-i-\alpha+
(a_l-a_1))f(i)+ \\[-5pt] \notag
& \sum_{i=a_l-a_1}^{a_l-1}f_R(n-i+(a_l-a_1))f(i),
\end{align}
where
\begin{equation}\label{CombinatorialfR}
f_R(q)=\sum_{z\in \mathcal{R}_{\mathbf{k}}(A,q-(a_l-a_1))}\pi(z).
\end{equation}
\end{theorem}

\begin{proof}
For every $z\in \mathcal{R}_{\mathbf{k}}(A,n-i)$, $\lambda(z)$ is an
$A$-restricted composition of $n-i$. Therefore, $\text{first}(\lambda(z))$ is
obviously greater than or equal to $a_1$, the least element of $A$. Thus, in
\eqref{Vni}, since $a_1\geq a_l-i$, $\text{first}(\lambda(z))$ will certainly
be greater than or equal to $a_l-i$  for $i \geq a_l-a_1$. Thus, we can rewrite
\eqref{Vni} as
 \begin{equation}\label{Vni1}
 V(n,i)=\sum_{z\in \mathcal{R}_{\mathbf{k}}(A,n-i)}\pi(z)=f_R(n-i+(a_l-a_1)).
 \end{equation}
for $i \geq a_l-a_1$. For $i<a_l-a_1$, $\text{first}(\lambda(z))\geq a_l-i$, or
equivalently, $i\geq a_l - \text{first}(\lambda(z))$. In this case, \eqref{Vni}
can be rewritten as
\begin{align}\label{Vni2}
 V(n,i)=&\sum_{\alpha\in A\setminus\{a_1\}}\left(\kappa(\alpha)\sum_{z\in  \mathcal{R}_{\mathbf{k}}(A,n-i-\alpha)} \pi(z)\right)  \\[2pt] \notag = & \sum_{\alpha\in A\setminus\{a_1\}}\kappa(\alpha)f_R(n-i-\alpha+(a_l-a_1)).
\end{align}
By feeding \eqref{Vni1} and \eqref{Vni2} into \eqref{Solution}, we derive
\eqref{DiophantineSolution}.
\end{proof}

It can easily be shown that
\begin{equation}
\sum_{z\in \mathcal{R}_{\mathbf{k}}(A,q-(a_l-a_1))}\pi(z)=\sum_{\{\mathbf{x}|\mathbf{x}\cdot\mathbf{v}^A=q-(a_l-a_1), \  \mathbf{x} \geq \mathbf{0}\}}P_o(\mathbf{k},\mathbf{x})C(\mathbf{x})
\end{equation}
in which
\begin{equation}\label{CombinatorialfR}
P_o(\mathbf{k},\mathbf{x})=\prod_{j=1}^{|\mathbf{x}|} k_j^{x_j}.
\end{equation}
Thus, given $A$ and $\mathbf{k}$, in order to calculate $f_R(q)$'s, we should
solve the Diophantine equation $\mathbf{x}\cdot\mathbf{v}^A=q-(a_l-a_1)$ with
the constraint $\mathbf{x} \geq \mathbf{0}$, which we call the resolvent of
\eqref{interp}.

Now, we apply our technique to two-term and three-term LHRCs.

\begin{corollary}\label{FirstCorollary}
The solution to the two-term LHRC
\begin{equation}\label{two-term}
f(n)=\kappa(a_1)f(n-a_1)+\kappa(a_2)f(n-a_2)
\end{equation}
is
\begin{align}\label{closed2term}
 f(n)= \kappa(a_2)\sum\limits_{i=0}^{a_2-a_1-1}f_R(n&-i-a_1)f(i)+
 \\[-5pt]  \notag &  \sum\limits_{i=a_2-a_1}^{a_2-1}f_R(n-i+(a_2-a_1))f(i),
\end{align}
where
\begin{align*}
&f_R(q)= \sum_{t=L}^{U} P_o(\mathbf{k}, \mathbf{x^{(t)}})C(\mathbf{x^{(t)}}),\\  \notag
& \mathbf{k}=[\kappa(a_1),\: \kappa(a_2)],\\  \notag
& \mathbf{x^{(t)}}=[rq'/g+t,sq'/g-ta_1/g],\\  \notag
& q'=q-(a_2-a_1),\\  \notag
& L=\lceil-rq'/a_2\rceil,\\  \notag
& U=\lfloor sq'/a_1\rfloor,\\ \notag
& g=\gcd(a_1,a_2),\\ \notag
\end{align*}
and $r$ and $s$ are Bezout coefficients for $a_1$ and $a_2$.
\end{corollary}

\begin{proof}
The resolvent of~\eqref{two-term} is the Diophantine equation
$a_1x_1+a_2x_2=q-(a_2-a_1)$ with the constraint $x_1, x_2 \geq 0$. The result
is then immediate from Theorem~\ref{SecondTheorem}.
\end{proof}

Three-term LHRCs are of special interest to researchers \cite{bib25, bib26}.
The following corollary shows how to solve these equations by the technique
described above.
\begin{corollary}\label{SecondCorollary}
The solution to the three-term LHRC
\begin{equation}\label{Three-Term}
f(n)=\kappa(a_1)f(n-a_1)+\kappa(a_2)f(n-a_2)+\kappa(a_3)f(n-a_3)
\end{equation}
is
\begin{align}\label{closed3term}
f(n)=& \kappa(a_2)\sum\limits_{i=a_3-a_2}^{a_3-a_1-1}f_R(n-i-a_2+(a_3-a_1))f(i)+\\  \notag
&\kappa(a_3)\sum\limits_{i=0}^{a_3-a_1-1}f_R(n-i-a_1)f(i)+\\  \notag
& \sum\limits_{i=a_3-a_1}^{a_3-1}f_R(n-i+(a_3-a_1))f(i),
\end{align}
where
\begin{align*}
&f_R(q)= \sum_{x_3=0}^{\lfloor
q'/a_3\rfloor}\sum_{t=L}^{U}P_o(\mathbf{k},\mathbf{x^{(t)}})C(\mathbf{x}^{(t)}),\\  \notag
& \mathbf{k}=[\kappa(a_1), \kappa(a_2), \kappa(a_3)],\\  \notag
& \mathbf{x^{(t)}}=[rq'/g+t, sq'/g-ta_1/g,x_3],\\   \notag
& q'=q-(a_3-a_1),\\  \notag
& \widehat{q}=q'-a_3x_3,\\  \notag
& L=\lceil-rq'/a_2\rceil,\\  \notag
& U=\lfloor sq'/a_1\rfloor,\\  \notag
& g=\gcd(a_1, a_2),
\end{align*}
and $r$ and $s$ are Bezout coefficients for $a_1$ and $a_2$.
\end{corollary}

\begin{proof}
The resolvent of~\eqref{Three-Term} is $a_1x_1+a_2x_2+a_3x_3=q-(a_3-a_1)$ with
the constraint $x_1,x_2,x_3\geq 0$. It can be solved through solving the
Diophantine equations $a_1x_1+a_2x_2=q-(a_3-a_1)-a_3x_3$ for every possible
$x_3$ in $\left[0:\lfloor q/a_3\rfloor\right]$. Theorem 2, then, implies
\eqref{closed3term}.
\end{proof}

\section{Relevant Problems}
\label{sec:relprob} In this section, we derive closed-form solutions for some
well-known cases of the $S$-restricted composition problem. We divide these
problems into two classes. The first class consists of those problems that are
directly solved through corresponding interpreters. The second class is
comprised of the problems that can be solved through finding the general term
of related integer sequences.

\subsection{Relevant $S$-restricted Composition
Problems}\label{subsec:relSrestricted} We employ the first and second
interpreters, i.e., \eqref{firstLHR} and \eqref{secondLHR}, to derive
closed-form solutions for some well-known $S$-restricted composition problems.
We begin with a well-known case of the $S$-restricted composition problem where
$S$ consists of two positive integers $a_1$ and $a_2$.

\begin{corollary}\label{cora1a2}
The number of compositions of $n$ into two distinct positive integers $a_1$ and
$a_2$ is obtained from
\begin{align}\label{a1a2}
R(\{a_1,a_2\},n)=\sum\limits_{h=0}^{\lfloor  \frac{a_2-a1-1}{a_1}\rfloor}&f_R(n-a_1(h+1))+
\\[-5pt] \notag & \sum\limits^{\lfloor\frac{a_2-1}{a_1}\rfloor}_{h=\lfloor\frac{a_2-a1-1}{a_1}\rfloor+1}f_R(n-a_1(h+1)+a_2),
\end{align}
where
\begin{align*}
&f_R(q)= \sum_{t=L}^{U} C(\mathbf{x^{(t)}}),\\ \notag
& \mathbf{x^{(t)}}=[rq'/g+t,sq'/g-ta_1/g],\\  \notag
& q'=q-(a_2-a_1),\\ \notag
& L=\lceil-rq'/a_2\rceil, \\ \notag
& U=\lfloor sq'/a_1\rfloor,\\  \notag
& g=\gcd(a_1,a_2),\\ \notag
\end{align*}
and $r$ and $s$ are Bezout coefficients for $a_1$ and $a_2$.
\end{corollary}

\begin{proof}
In this case, it is evident that the first interpreter has fewer terms than the
second interpreter. Indeed, the first interpreter is $R(\{a_1,a_2\},n) =
R(\{a_1,a_2\},n-a_1) + R(\{a_1,a_2\},n-a_2)$ for $n
>a_2-1$. Moreover, $R(\{a_1,a_2\},i)$ is equal to $1$ for
$0\leq i\leq a_{2}-1$ if $n\mod a_1=0$ and equals $0$ otherwise. Thus, the
solution is \eqref{a1a2}.
\end{proof}
Equation~\eqref{a1a2} gives a closed-form solution to the general problem whose
special cases have been studied in \cite{bib13,bib14,bib15}. The same equation
can also be used to calculate the number of compositions of $n$ into the
positive integers which are congruent to a given integer $r$ modulo the given
integer $m$. In doing so, it suffices to calculate the number of compositions
of $n-r$ into $r$ and $m$. In this way, \eqref{a1a2} gives a closed-form
solution to the general problems whose special cases have been studied in
\cite{bib35,bib43}.

\begin{corollary}
The number of compositions of a positive integer $n$ into the positive integers
which are less than or equal to $m$ is obtained from
\begin{equation}\label{1-m}
R(\left[1:m\right], n)=\left( 2 \right)^{m-1}f_R(n)-f_R(n-1)-\sum\limits_{i=1}^{m-1}\left( 2 \right)^{i-1}f_R(n-i-1),
\end{equation}
where
\begin{align*}
&f_R(q)= \sum_{t=L}^{U}P_o(\mathbf{k},\mathbf{x^{(t)}})C(\mathbf{x^{(t)}}),\\  \notag
& \mathbf{x^{(t)}}=[-mq'+(m+1)t,\:q'-t],\\ \notag
& \mathbf{k}=[2,-1],\\ \notag
& L=\lceil mq'/(m+1)\rceil,\\ \notag
& U=q',\\  \notag
& q'=q-m.
\end{align*}
\end{corollary}
\begin{proof}
In this case, the second interpreter is preferred. It is
\begin{equation*}\label{mfibonacci}
R(\left[1:m\right], n)=2R(\left[1:m\right], n-1)-R(\left[1:m\right],n-(m+1))
\end{equation*}
with initial conditions $R(\left[1:m\right], 0)=1$ and $R(\left[1:m\right],
i)=2^{i-1}$ for $1\leq i\leq m-1$. Thus, in order to calculate
$R(\left[1:m\right])$, n), we put $a_1=1, a_2=m+1, \kappa(a_1)=2,
\kappa(a_2)=-1$ and apply corollary \ref{cora1a2}. It leads to \eqref{1-m}.
\end{proof}

\begin{corollary}
The number of compositions of a positive integer $n$ into positive integers
greater than or equal to $m$ satisfies
\begin{equation*}
R([m:n],n) =f_R(n-m)+\sum\limits_{i=0}^{m-2}f_R(n-m-i-1)
 \end{equation*}
in which
\begin{align*}
f_R(q)=&\sum_{t=L}^{U}C(\mathbf{x^{(t)}}),\\ \notag
& \mathbf{x^{(t)}}=[-(m-1)q'+tm, q'-t],\\ \notag
& L=\lceil (m-1)q'/m\rceil,\\ \notag
& U=q',\\ \notag
& q'=q-(m-1).
\end{align*}
\end{corollary}
\begin{proof}
It is known that the number of compositions of $n$ into parts greater than or
equal to $m$ is equivalent to the number of compositions of $n-m$ into $1$ and
$m$.
\end{proof}

The following corollary shows how we can calculate the number of compositions
of $n$ in which the integer $m$ does not appear. This problem has already been
studied in \cite{bib16}, but no closed-form solution has yet been presented.
\begin{corollary}\label{withoutm}
For an integer $m > 1$, the number of compositions of $n$ where no part is $m$
is obtained from
\begin{align*}
R([1:n]\setminus{\{m\}},n)=\left(2^{m-1}-1\right)&f_R(n)+ \\[-5pt] & \left(2^{m-2}\right)f_R(n-m)-\sum\limits_{i=2}^{m-1}\left(2^{i-2}\right)f_R(n-i),
\end{align*}
where
\begin{align*}
& f_R(q)= \sum_{x_3=0}^{\lfloor q'/(m+1)\rfloor}\sum_{t=L}^{U}P_o(\mathbf{k},\mathbf{x^{(t)}})C(\mathbf{x^{(t)}}),\\ \notag
& q'=q-m,\\ \notag
& \mathbf{x^{(t)}}=[-(m-1)\widehat{q}+tm, \widehat{q}-t, x_3],\\ \notag
& \widehat{q}=q'-(m+1)x_3, \\
& L=\lceil (m-1)\widehat{q}/m\rceil,\\ \notag
& U=\widehat{q}.
\end{align*}
\end{corollary}
\begin{proof}
For this case, the second interpreter is
\begin{align*}\label{withoutm}
 R([1:n]\setminus{\{m\}},n)=& 2R([1:n]\setminus{\{m\}}, n-1)- \\ & R([1:n]\setminus{\{m\}},n-m)+\\ &R([1:n]\setminus{\{m\}},n-(m+1)).
\end{align*}
For $m>1$, we have $k_{a_1}=2$, $k_{a_2}=-1$, $k_{a_3}=1$, $a_1=1$, $a_2=m$,
and $a_3=m+1$ with initial conditions
$R([1:n]\setminus{\{m\}},0)=R([1:n]\setminus{\{m\}},1)=1$,
$R([1:n]\setminus{\{m\}},n)=2^{n-1}$ for $1 < n \leq m-1$, and
$R([1:n]\setminus{\{m\}},m)=2^{m-1}-1$. It has previously been shown that
$R([1:n]\setminus{\{1\}},n)$ is equal to the $(n-1)$th Fibonacci
number\cite{bib8}.
\end{proof}

For $m=1$, the equation in the proof of Corollary~\ref{withoutm} is converted
to $R([2:n],n)=R([2:n],n-1)+R([2:n],n-2)$ with initial conditions
$R([2:n],0)=1$ and $R([2:n],1)=0$. In this case, the solution is
\begin{equation*}
R([2:n],n)=\sum_{t=\lceil\frac{n-1}{2}\rceil}^{n-1}{t \choose n-1-t}
\end{equation*}
for $n \geq 2$ which is equal to the $(n-1)$th Fibonacci number
\cite{bib8,bib104}.

\begin{corollary}
The number of compositions of $n$ in which no part is in $[m_1:m_2]$ satisfies
\begin{align*}
R([1:n]\setminus{[m_1:m_2]},n)= & f_R(n-1)+R([1:m_1-1], m_2)f_R(n)+\\ \notag
& \quad \ \sum\limits_{i=1}^{m_1-1}(2^{i-1})f_R(n-i-1)+ \\ \notag
& \quad \ \sum\limits_{i=m_1}^{m_2-1}R([1:m_1-1],i)f_R(n-i-1)- \\ \notag
& \sum\limits_{i=m_2-m_1+1}^{m_1-1}(2^{i-1})f_R(n-i-m_1+m_2)+\\ \notag
& \quad \ \sum\limits_{i=m_1}^{m_2-1}R([1:m_1-1],i)f_R(n-i-m_1+m_2),
\end{align*}
where
\begin{align*}
& f_R(q)= \sum_{x_3=0}^{\lfloor
q'/(m_2+1)\rfloor}\sum_{t=L}^{U}P_o(\mathbf{k},\mathbf{x^{(t)}})C(\mathbf{x^{(t)}}),\\ \notag
& q'=q-m_2,\\ \notag
& L=\lceil (m_1-1)\widehat{q}/m\rceil,\\ \notag
& U=\widehat{q},\\ \notag
& \mathbf{x^{(t)}}=[-(m_1-1)\widehat{q}+tm_1, \widehat{q}-t, x_3],\\ \notag
& \widehat{q}=q'-(m_2+1)x_3,\\ \notag
& \mathbf{k}=[2, -1, 1].
\end{align*}
\end{corollary}

\begin{proof}
In this case, the second interpreter is
\begin{align}\label{m1m2}
R([1:n]\setminus{[m_1:m_2]},n)=&2R([1:n-1]\setminus{[m_1:m_2]},n-1)-\\ \notag
&R([1:n-m_1]\setminus{[m_1:m_2]},n-m_1)+\\ \notag &R([1:n-(m_2+1)]\setminus{[m_1:m_2]},n-(m_2+1)).
\end{align}
If $m_2=m_1+1$, \eqref{m1m2} is solved in the same way as the second
interpreter of compositions without $m$ where $m$ is replaced with $m_1$. If
$m_2 \geq m_1+2$ (which results in $m_2-2>m_1-1$), we have
$\mathbf{k}=[2,-1,1]$ and $\mathbf{a}=[1,m_1,m_2]$. The initial conditions are
also $R([1:n]\setminus{[m_1:m_2]},0)=1$,
$R([1:n]\setminus{[m_1:m_2]},i)=2^{i-1}$ for $1\leq i\leq m_1-1$, and
$R([1:n]\setminus{[m_1:m_2]},i)=R([1:m-1],i)$ for $m_1\leq i\leq m_2$.
\end{proof}

\subsection{Relevant Sequences}\label{subsec:relseq}
This section is devoted to finding the general term of a number of known
sequences by applying the technique developed in this paper. We begin with the
$m$-Fibonacci sequence, which is known as the Fibonacci, the Tribonacci, and
the Tetranacci sequence for $m=2$, $m=3$, and $m=4$, respectively.

\begin{corollary} The $n$th element of the $m$-Fibonacci sequence defined by the LHRC $F^m(n)=
\sum_{i=1}^mF^m(n-i)$ with the initial conditions $F^m(i)=0$ for $0\leq i \leq
m-2$ and $F^m(m-1)=1$ is
\begin{equation}\label{ttt}
F^m(n)=f_R(n)+f_R(n-m),
\end{equation}
where
\begin{align*}
f_R(q)=& \sum_{t=L}^{U}  P_o(\mathbf{k},\mathbf{x^{(t)}})C(\mathbf{x^{(t)}}),\\ \notag
& \mathbf{x^{(t)}}=[-mq'+(m+1)t,q'-t],\\ \notag
& L=\lceil mq'/(m+1)\rceil,\\ \notag
& U=q', \\ \notag
& q'=q-m, \\ \notag
& \mathbf{k}=[2,-1].
\end{align*}
\end{corollary}

\begin{proof}
The $n$th element of the sequence satisfies
\begin{equation*}
F^m(n)-F^m(n-1)= F^m(n-1)-F^m(n-(m+1))
\end{equation*}
or, equivalently,
\begin{equation}\label{mfibonacci2}
F^m(n)=2F^m(n-1)-F^m(n-(m+1)).
\end{equation}
We have solved two-term LHRCs such as \eqref{mfibonacci2} in
Section~\ref{sec:solvLHRC}. Thus, we can use \eqref{closed2term} where the
initial conditions are $f(i)=0$ for $0\leq i \leq m-2$ and $f(m-1)=f(m)=1$. The
result is then \eqref{ttt}.
 \end{proof}

As an instance, $F^2(n)=f_R(n)+f_R(n-2)$ represents the $n$th element of the
well-known Fibonacci sequence where $f_R(q)=\sum_{t=L}^{U} P_o(\mathbf{k},
\mathbf{x^{(t)}})C(\mathbf{x^{(t)}})$, $L=\lceil 2q'/3\rceil$, $U=q'$,
$\mathbf{x}^{(t)}=[-2q'+3t, q'-t]$ and $q'=q-2$.

\begin{corollary}
The $n$th element of the Lucas sequence defined by
\begin{equation*}
L_u(n)=L_u(n-1)+L_u(n-2); \qquad  L_u(0)=1,L_u(1)=3
\end{equation*}
is obtained from
\begin{align}\label{Lucas}
 L_u(n)=F_{n-1}+3F_n=3\sum\limits_{t=0}^{\lfloor\frac{n-1}{2}\rfloor}{n-1-t \choose t}+\sum\limits_{t=0}^{\lfloor\frac{n-2}{2}\rfloor}{n-2-t \choose t},
 \end{align}
where $F_n=F^2(n)=\sum_{t=0}^{\lfloor\frac{n-1}{2}\rfloor}{n-1-t \choose t}$ is
the $n$th Fibonacci number.
\end{corollary}

\begin{proof}
The LHRC $f(n)=f(n-1)+f(n-2)$ is solved as follows.
\begin{equation*}
f(n)=f(0)\sum\limits_{t=\lceil\frac{n-2}{2}\rceil}^{n-2}{t \choose n-2-t}+f(1)\sum\limits_{t=\lceil\frac{n-1}{2}\rceil}^{n-1}{t \choose n-1-t}.
\end{equation*}
Since the $n$th Fibonacci number is equal to
$F_n=\sum_{t=\lceil\frac{n-1}{2}\rceil}^{n-1}{t \choose n-1-t}$, the above
equation can be rewritten as
\begin{equation}\label{LucaFib}
f(n)=F_{n-1}f(0)+F_nf(1).
\end{equation}
As $f(0)=L_u(0)=1$ and $f(1)=L_u(1)=3$, \eqref{LucaFib} implies \eqref{Lucas}.
\end{proof}

\begin{corollary}
The $n$th element of the Padovan sequence represented by the LHRC
$P^d_n=P^d_{n-2}+P^d_{n-3}$ with the initial conditions $P^d_0=0$, $P^d_1=1$,
and $P^d_2=0$ is
\begin{equation*}
P^d_n=\sum_{t=\lceil\frac{(n-1)}{3}\rceil}^{\lfloor\frac{(n-1)}{2}\rfloor}{t\choose
n-1-2t}.
\end{equation*}
Moreover, the $n$th element of the Perrin sequence represented by the LHRC
$P^r_n=P^r_{n-2}+P^r_{n-3}$ with the initial conditions $P^r_0=3$, $P^r_1=0$,
and $P^d_2=2$ is equal to $2P^d_{n-1}+3P^d_{n-2}$.
\end{corollary}

\begin{proof}
The LHRC $f(n)=f(n-2)+f(n-3)$ satisfied by the $n$th element of the Padovan and
Perrin sequences is solved as follows:
\begin{align*}
f(n)=& f(0)\sum\limits_{t=\lceil\frac{(n-3)}{3}\rceil}^{\lfloor\frac{(n-3)}{2}\rfloor}{t\choose n-3-2t}+f(1)\sum\limits_{t=\lceil\frac{(n-1)}{3}\rceil}^{\lfloor\frac{(n-1)}{2}\rfloor}{t\choose n-1-2t}+\\[2pt]
\notag & f(2)\sum\limits_{t=\lceil\frac{(n-2)}{3}\rceil}^{\lfloor\frac{(n-2)}{2}\rfloor}{t\choose n-2-2t}
\end{align*}
In the Padovan sequence, $f(0)=0$, $f(1)=1$ and $f(2)=0$. Thus, the $n$th term
of the Padovan sequence can be calculated as
$P^d_n=\sum_{t=\lceil(n-1)/3\rceil}^{\lfloor(n-1)/2\rfloor}{t\choose n-1-2t}$.
Furthermore, the solution to $f(n)=f(n-2)+f(n-3)$ is $f(n)=P^d_{n-2}f(0)+P^d_n
f(1)+P^d_{n-1}f(2)$. By putting $f(0)=P^r_0=3$, $f(1)=P^r_1=0$ and
$f(2)=P^r_2=2$ in this equation, we obtain $P^r_n=2P^d_{n-1}+3P^d_{n-2}$.
\end{proof}

\begin{corollary}
The $n$th element of the Pell sequence represented by the LHRC
$P^l_n=2P^l_{n-1}+P^l_{n-2}$ with the initial conditions $P^l_0=0$ and
$P^l_1=1$ is
\begin{equation*}
P^l_n=\sum_{t=\lceil (n-1)/2\rceil}^{n-1}2^{-n+1+2t}{t\choose n-1-t}.
\end{equation*}
Moreover, the $n$th element of the Pell-Lucas sequence represented by the LHRC
$P^L_n=2P^L_{n-1}+P^L_{n-2}$ with the initial conditions $P^L_0=P^L_1=2$ is
equal to $2P^l_n+2P^l_{n-1}$.
\end{corollary}
\begin{proof}
Both Pell and Pell-Lucas sequences satisfy $f(n)=2f(n-1)+f(n-2)$ which is
solved as follows:
\begin{align*}
f(n)=& f(0)\sum\limits_{t=\lceil\frac{n-2}{2}\rceil}^{n-2}2^{-n+2+2t}{t\choose n-2-t}+\\[2pt]
& f(1)\sum\limits_{t=\lceil\frac{n-1}{2}\rceil}^{n-1}2^{-n+1+2t}{t\choose n-1-t}.
\end{align*}
The initial conditions are $P^l_0=0$ and $P^l_1=1$ for the Pell sequence and
$P^L_0=P^L_1=2$ for the Pell-Lucas sequence.
\end{proof}

Finally, we refer to a number of proven bijections between the elements of the
sequences studied in this section and solutions to some cases of the
$S$-restricted composition problem. They indeed indicate that the study of such
sequences is quite helpful in finding the number of $S$-restricted compositions
of a given integer.
\begin{enumerate}
\item The number of compositions of $n$ into $2$ and $3$ is the $(n-2)$th
    term of the Padovan sequence.
\item The number of compositions of $n$ in which no part is $2$ is the
    $(n-2)$th term of the Padovan sequence.
\item The number of compositions of $n$ with summands congruent to $2$ modulo
    $3$ is the $(n-4)$th term of the Padovan sequence.
\item The number of compositions of $n$ with odd summands greater than $1$ is
    the $(n-5)$th term of the Padovan sequence.
\item The number of compositions of $n$ into odd parts is the $n$th term of
    the Fibonacci sequence.
\end{enumerate}

\section{Conclusion}
\label{sec:conclusion} We propose a technique for solving LHRCs through solving
Diophantine equations. The technique can be applied to the LHRCs having a small
number of terms. We also show that the $S$-restricted composition problem is
reduced to the problem of solving particular LHRCs which we call interpreters.
Some well-known cases of the $S$-restricted composition problem are studied and
it is demonstrated that deciding on an appropriate interpreter, like the ones
proposed in this paper, can lead to a closed-form solution to the problem. By
using the same technique, we also find closed-form formulas for the general
term of a number of well-known integer sequences. Deriving appropriate
interpreters for the cases not studied in this paper deserves further research.

%\acknowledgements \label{sec:ack} At the end of the manuscript, right before
%the bibliography you might want to place an acknowledgment. This can be easily
%done by using the command \verb!\acknowledgements! as you can see here.

\nocite{*}
\bibliographystyle{abbrvnat}
%% use the following instead if you encounter problems
%%\bibliographystyle{alpha}
\bibliography{SRestricted}
\label{sec:biblio}

\end{document}